\documentclass{amsart}

\usepackage{amsmath}
\usepackage{amssymb}
\usepackage{amsthm}
\usepackage{latexsym}
\usepackage{mathrsfs}
\usepackage{stmaryrd}
\usepackage{verbatim}
\usepackage{xcolor}

\calclayout

\newtheorem{theorem}{Theorem}
\newtheorem{definition}[theorem]{Definition}
\newtheorem{lemma}[theorem]{Lemma}

\newtheorem{corollary}[theorem]{Corollary}

\theoremstyle{definition}

\renewcommand{\AA}{\mathbb{A}}

\newcommand{\PP}{\mathbb{P}}

\newcommand{\RR}{\mathbb{R}}

\newcommand{\Trop}{\operatorname{Trop}}

\begin{document}
\title[Avoidance loci and tropicalizations of real bitangents]{Avoidance loci and tropicalizations of real bitangents to plane quartics}
\author {Hannah Markwig}
\address {Universit\"at T\"ubingen, Fachbereich Mathematik, Auf der Morgenstelle 10, 72076 T\"ubingen, Germany }
\email {hannah@math.uni-tuebingen.de}
\author {Sam Payne}
\address {Department of Mathematics, University of Texas at Austin, 2515 Speedway, PMA 8.100, Austin, TX 78712, USA}
\email {sampayne@utexas.edu}
\author {Kris Shaw}
\address {Department of Mathematics, University of Oslo, Postboks 1053, Blindern, 0316 Oslo, Norway}
\email {krisshaw@math.uio.no}

\bibliographystyle{alpha}

\begin{abstract}
We compare two partitions of real bitangents to smooth plane quartics into sets of 4: one coming from the closures of connected components of the avoidance locus and another coming from tropical geometry.  When both are defined, we use the Tarski principle for real closed fields in combination with the topology of real plane quartics and the tropical geometry of bitangents and theta characteristics to show that they coincide.
\end{abstract}

\maketitle

\bigskip

A smooth plane quartic curve over an algebraically closed field of characteristic not equal to 2 has precisely 28 bitangent lines, whose geometry is closely related to that of the 27 lines on a cubic surface. The geometric and arithmetic properties of these structures over non-closed fields are an important testing ground and have recently featured prominently, for instance, in the development of $\AA^1$-enumerative geometry and exploration of its connections to real and tropical algebraic geometry \cite{KW20, LV21, MPS23}.

A smooth plane quartic curve $X$ over the real numbers has either $4$, $8$, $16$, or $28$ real bitangents, precisely $4$ of which are disjoint from $X(\RR)$ \cite{Zeuthen}. The \emph{avoidance locus} of $X$ is the set of all real lines in $\PP^2$ that are disjoint from $X(\RR)$.  It is open in the locally Euclidean topology on the real dual projective plane $\check\PP^2(\RR)$, and each of its connected components contains precisely $4$ real bitangents in its closure \cite[Corollary~2.4]{kum:19}. Moreover, each real bitangent to $X$ is in the closure of a unique component of the avoidance locus; this follows from the correspondence between bitangents and odd theta characteristics together with the discussion of definite forms associated to odd theta characteristics in  \cite[\S2]{kum:19}.  Thus, the real bitangents are canonically partitioned into sets of $4$.

The number of components of the avoidance locus depends on the topology of $X(\RR) \subset \PP^2(\RR)$.  Let $s$ be the number of connected components of $X(\RR)$, and let $a$ be $1$ if $\PP^2(\RR) \smallsetminus X(\RR)$  is connected, and $0$ otherwise. Then the number of connected components of the avoidance locus is $2^{s-1} - 1 + a$; see the second sentence following \cite[Definition~4.1]{kum:19}. The fact that the number of connected components is $1$, $2$, $4$ or $7$ then follows from the topological classification of smooth plane quartics; see, e.g., \cite[\S7]{GrossHarris81}.  Analogous statements hold for smooth quartic curves over any real closed field, by the Tarski principle. Note, however, that real closed fields other than $\RR$ are totally disconnected in the order topology. One must use an appropriate notion of definable connectedness, e.g., as in \cite[Chapter~6]{vdD}, in place of topological connectedness, to characterize ``connected components" of semialgebraic sets such as avoidance loci. A definable set is definably connected if it cannot be written as a disjoint union of two non-empty definable open sets. 

\medskip

Tropical geometry also gives a natural partition of bitangents into sets of $4$, for plane quartics over valued fields with smooth tropicalization.  Moreover, tropical geometers have observed computationally that, for plane quartics over real closed valued fields with sufficiently general tropicalization, either all or none of the bitangents in each group of $4$ are real, i.e., rational over the real closed base field \cite{GP21, CM23}.  In this short note, we prove that this tropically observed phenomenon holds in a large and natural level of generality, for all curves with smooth tropicalization. The key step in the proof is showing that the tropical partition agrees with the partition by closures of components of the avoidance locus, whenever both are defined.  

\medskip

Let $X$ be a smooth plane quartic curve over a real-closed field $K$ equipped with a nontrivial valuation. For instance, $K$ could be the field of real Puiseux series $\RR\{\!\{t\}\!\}$ with its $t$-adic valuation.

\begin{definition}
The \emph{avoidance locus} of $X$ is the subset of $\check\PP^2(K)$ parametrizing lines that are disjoint from $X(K)$.
\end{definition}

\noindent By the Tarski principle, i.e., by elimination of quantifiers in the first order theory of real closed fields, the avoidance locus of $X$ has $1$, $2$, $4$, or $7$ definably connected components, each of which has exactly $4$ $K$-rational bitangents in its closure. Correspondingly, the curve $X$ has $4$, $8$, $16$, or $28$ $K$-rational bitangents, each of which is contained in the closure of a unique  component of the avoidance locus.

\medskip

The valuation on $K$ induces a tropicalization map, and we assume that $X \subset \PP^2$ is \emph{tropically smooth}, i.e., the dual Newton sudivision of a defining equation for $X$, induced by the valuations of the coefficients, is a unimodular triangulation of the $4$-fold dilation of the standard simplex.  Then $\Trop(X)$ has precisely $7$ equivalence classes of tropical bitangents \cite{BLMPR}.  Here, a tropical bitangent is a tropical line whose intersection with $\Trop(X)$ is either connected, or else has two connected components, each of tropical multiplicity 2, and two tropical bitangents are equivalent if they correspond to linearly equivalent theta characteristics  \cite[Definitions 3.1 and 3.8]{BLMPR}.

 Each equivalence class of tropical bitangents contains the tropicalization of precisely $4$ algebraic bitangents defined over the algebraic closure of $K$.  This is proved by studying the tropicalizations of even and odd theta characteristics \cite{JensenLen18}, and using the bijective correspondence between the 28 odd theta characteristics and the 28 bitangents of a smooth plane quartic. The odd theta characteristic corresponding to a bitangent to a plane quartic is obtained by taking the sum of the two intersection points. 

\medskip

Prior to the present work, it was known under some additional genericity hypotheses on $\Trop(X)$, via a case-by-case combinatorial and computational analysis that either $0$ or all $4$ of the geometric bitangents in each of these tropical equivalence classes are $K$-rational \cite{CM23, MPS23}  and that the $1$, $2$, $4$, or $7$ tropical bitangent classes that lift to $K$-rational bitangents behave well under suitable tropical deformations \cite{GP21}.  Our main results are as follows:

\begin{theorem} \label{thm:main}
Suppose that $X$ is a plane quartic over $K$ and $\Trop(X)$ is smooth.  Then each definably connected component of the avoidance locus of $X$ tropicalizes into a distinct equivalence class of tropical bitangents to $\Trop(X)$.  
\end{theorem}

\begin{corollary}
Let $S$ be an equivalence class of tropical bitangents to $\Trop(X)$.  Then:
\begin{enumerate}
\item Either $S$ contains the tropicalization of a definably connected component of the avoidance locus along with the tropicalizations of exactly $4$ $K$-rational bitangents, or 
\item $S$ does not meet the tropicalization of the avoidance locus and does not contain the tropicalization of a $K$-rational bitangent.
\end{enumerate}
\end{corollary}

In this way, tropicalization induces a bijection between the semialgebraic path connected components of the avoidance locus to $X$ and the tropical bitangents of $\Trop(X)$ that lift to $K$-rational bitangents.  Moreover, the grouping of $K$-rational bitangents given by containment in the closures of the definably connected components of the avoidance locus agrees with that induced from tropical equivalence classes of bitangents.  

For a line $L$ in $\PP^2$, let $[L]$ denote the corresponding point in $\check \PP^2$.

\begin{lemma} \label{lem:bitangent}
Suppose $[L] \in \check \PP^2(K)$ is in the avoidance locus of $X$.  Then $\Trop(L)$ is a tropical bitangent to $\Trop(X)$.
\end{lemma}

\begin{proof}
Since $[L]$ is in the avoidance locus of $X$, the geometric intersection of $L$ with $X$ consists of two pairs of conjugate points. Each such pair tropicalizes to one point. Thus the intersection multiplicity of $\Trop(L)$ with $\Trop(X)$ along each connected component of their intersection is either $2$ or $4$, depending on whether it contains the tropicalization of $1$ or $2$ of these pairs, by \cite[Theorem~6.4]{OR13}.  In particular, $\Trop(L)$ is a tropical bitangent of $\Trop(X)$.
\end{proof}

\begin{lemma} \label{lem:convex}
The preimage under tropicalization of each definably connected component of the avoidance locus of $X$ is a union of two convex open cones in $K^3 \smallsetminus \{0\}$.
\end{lemma}

\begin{proof}
By the Tarski principle, it suffices to prove this statement for a smooth plane quartic $Y$ over $\RR$.  The avoidance locus is open, so its preimage $U$ in $K^3 \smallsetminus \{0\}$ is an open cone. We claim that $U$ is a disjoint union of two convex connected components, each of which is the negative of the other.  To see this, note that each connected component of $U$  determines an orientation on the ovals of $Y(\RR)$, via the correspondence between points in the avoidance locus and definite nowhere vanishing differentials; see Corollary~2.2 and the discussion following Definition~4.1 in \cite{kum:19}.  The lemma follows, since convex combinations in $K^3 \smallsetminus \{0\}$ correspond to convex combinations of these differentials.
\end{proof}

\begin{proof}[Proof of Theorem~\ref{thm:main}]
Let $S$ be a definably connected component of the avoidance locus of $X$.  Then the preimage of $S$ in $K^3 \smallsetminus \{0\}$ is a union of two convex cones, each of which is the negative of the other, by Lemma~\ref{lem:convex}. Now $\Trop(S)$ is the image of either one of these components. It follows that $\Trop(S)$ is tropically convex, and hence connected.  Note that very small deformations of a line, with respect to the nonarchimedean norm on $K$, do not change the tropicalization, so $\Trop(\overline S) = \Trop(S)$.  By Lemma~\ref{lem:bitangent}, every tropical line in $\Trop(\overline S)$ is a bitangent, so all $4$ of the $K$-rational bitangent lines in $\overline S$ tropicalize into the same equivalence class of tropical bitangents.  

By \cite{BLMPR, JensenLen18}, each equivalence class of tropical bitangents to $\Trop(X)$ contains the tropicalizations of exactly $4$ geometric bitangents.  Thus, distinct definably connected components of the avoidance locus must tropicalize into distinct equivalence classes of tropical bitangents, as claimed.
\end{proof}

In our proof of Theorem~\ref{thm:main}, we have used convexity of definably connected components of the avoidance locus.  This relates to the observed phenomenon that equivalence classes of tropical bitangents are very often tropically convex \cite{CM23}.  Alternatively, one could argue that tropicalizations of definably connected semialgebraic sets are connected, using \cite[Corollary~6.10]{JSY}.

\medskip

\noindent \textbf{Acknowledgments.} 
HM supported in part by DFG-grant MA 4797/9-1. SP supported in part by NSF grants DMS--2001502 and DMS--2053261.  KS supported in part by the Trond Mohn Foundation project ``Algebraic and topological cycles in complex and tropical geometry" and the Center for Advanced Study Young Fellows Project ``Real Structures in Discrete, Algebraic, Symplectic, and Tropical Geometries". 

\bibliography{AvoidanceLoci}

\end{document}